\documentclass[joc]{ipart}

\usepackage{amsmath,amsthm} 
\usepackage{amssymb}
\usepackage{mathtools}
\usepackage{relsize}
\usepackage{graphicx}
\usepackage{rotating}
\usepackage{pdflscape}
\newtheorem{theorem}{Theorem}[section]

\newtheorem{corollary}[theorem]{Corollary}
\newtheorem{proposition}[theorem]{Proposition}
\newtheorem{lemma}[theorem]{Lemma}

\numberwithin{equation}{section}
\usepackage{lineno,hyperref}
\modulolinenumbers[5]

\pubyear{XXXX}
\volume{X}
\issue{X}
\firstpage{1}
\lastpage{19}

\begin{document}

\begin{frontmatter}
\title[Directed Paths in a Slit]{Skew Schur Function Representation of Directed Paths in a Slit}

\begin{aug}
    \author{\fnms{Anum} \snm{Khalid}
    \ead[label=e1]{anum.khalid@qmul.ac.uk}}
    \address{School of Mathematical sciences, \\Queen Mary University of London,\\ Mile end Road, London E1 4NS.\\
             United Kingdom\\
             \printead{e1}}
    \and
    \author{\fnms{Thomas} \snm{Prellberg}
            \ead[label=e3]{t.prellberg@qmul.ac.uk}%
            \ead[label=u1,url]{http://www.qmul.ac.uk}}
    \address{School of Mathematical sciences, \\ Queen Mary University of London, \\ Mile end Road, London E1 4NS.\\
             United Kingdom\\
             \printead{e3}\\
             \printead{u1}}
\end{aug}

\begin{abstract}
In this work, we establish a general relationship between the enumeration of weighted directed paths and skew Schur functions, extending work by Bousquet-M\'elou,
who expressed generating functions of discrete excursions in terms of rectangular Schur functions.
\end{abstract}



\end{frontmatter}

\section{Introduction and Statement of Results}

We consider the enumeration of directed paths constrained to lie within a strip, with steps taken from a finite set of allowed steps having prescribed weights. Previous work in this context on bounded excursions found generating function expressions in terms of rectangular Schur functions \cite{Bousquet2006}. In related work \cite{2013arXiv1303.2724B}, bounded meanders were studied using a transfer matrix approach. Both meanders and excursions start at height zero, but while excursions are restricted to also end at height zero, meanders have no such endpoint restriction. In this paper, we extend these results by considering bounded paths starting and ending at arbitrary given heights. We express their generating functions in terms of skew Schur functions, and provide an expansion of these skew Schur functions in terms of a linear combination of Schur functions.
Related work has appeared in \cite{MR2735330}; in Theorem 4 of that paper the authors arrive at an expression using a generating variable for the endpoint position. An alternative proof of our Corollary \ref{schurcorol} could in principle be obtained from that expression.\begin{figure}[ht!]
\centering
\includegraphics[width=0.99\textwidth]{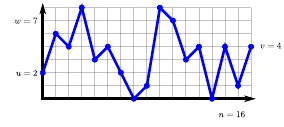}
\caption{Directed path of length $n=16$ with northeast steps $A=\{1,3,4,6\}$ and southeast steps $B=\{1,2,3,4\}$ ($\alpha=6$ and $\beta=4$) in a slit of width $w=7$, starting at height $u=2$ and ending at height $v=4$. 
}
\label{gd}
\end{figure}

Consider a directed $n$-step path in the slit $ \mathbb{Z} \times \{0,1,\cdots,w\}$ of width $w$, starting at point $(0,u)$ and ending at point $(n,v)$, taking its steps from $\{1\} \times S$, where $S \subset \mathbb{Z}$ is a finite set. For simplicity we call $S$ the step set. Figure \ref{gd} shows such a path. 
We separate the step set $S$ into sets of up and down steps by defining $A=S \cap \mathbb{Z}_0^+$ and $B=-(S \backslash A)$, where we have included the horizontal step in the set $A$.  Every up step of height $a\in A$ is weighted by a weight $p_a$, and every down step of height $b\in B$ is weighted by a weight $q_b$. We denote the maximum of $A$ and $B$ by $\alpha$ and $\beta$, respectively, and assume that the weights $p_{\alpha}$ and $q_{\beta}$ are nonzero. The weight $\omega_{\varphi}$ of a path $\varphi$ is then the product of the weights of all the steps in the path. The introduction of weights implies that by assigning a weight of zero to any integer not appearing in $S$ we can without loss of generality consider $A=\{0,1,\ldots,\alpha\}$ and $B=\{1,2,\ldots,\beta\}$. 

Given a step set $S$ and associated step weights, let $\Omega_{(u,v),n}^{w,\alpha,\beta}$ be the set of directed $n$-step paths in a strip of width $w$ starting at $(0,u)$ and ending at $(n,v)$. The main object of this paper is the generating function of directed weighted paths
\begin{equation}
G_{(u,v)}^{w,\alpha,\beta}(t)=\sum_{n=0}^\infty t^n\sum_{\varphi\in\Omega_{(u,v),n}^{w,\alpha,\beta}}\omega_{\varphi}\;.
\end{equation}

Having a finite strip width automatically implies that the generating function is rational, as the enumeration problem can be cast as a random walk problem on a finite graph and thus the generating function can be found from its transition matrix. This approach has for example been followed in \cite{2013arXiv1303.2724B}. One can easily deduce some complexity results, such as giving upper bounds on the degree of the polynomials appearing in the rational generating function, and also compute $G_{(u,v)}^{w,\alpha,\beta}(t)$ for specific parameter values. However, computing a general expression is considerably more difficult, with only some results available for meanders, {\it i.e.}~$G_{(0,v)}^{w,\alpha,\beta}(t)$ \cite{2013arXiv1303.2724B}. Following along ideas from \cite{Bousquet2006}, where an explicit expression was obtained for excursions, {\it i.e.}~$G_{(0,0)}^{w,\alpha,\beta}(t)$, our approach enables us to provide a general solution for $G_{(u,v)}^{w,\alpha,\beta}(t)$ in Theorem \ref{theorem1}.

\begin{theorem}
The generating function $G_{(u,v)}^{w,\alpha,\beta}(t)$ of directed weighted paths is given by
\begin{equation}
G_{(u,v)}^{w,\alpha,\beta}(t)=\frac{(-1)^{1-\alpha}}{tp_{\alpha}}\frac{s_{(w^{\alpha},u,0^{\beta-1})/(v,0^{\alpha+\beta-1})}(\bar{z})}{s_{((w+1)^{\alpha},0^{\beta})}(\bar{z})}\;,
\label{skewresult}
\end{equation}
where $\bar z$ are the $\alpha+\beta$ roots of
\[
K(t,z)=1-t\sum_{a \in A}p_az^a-t\sum_{b\in B}q_bz^{-b}\;,
\]
and $s_{\lambda/\mu}(z)$ is a skew Schur function.
\label{theorem1}
\end{theorem}

Schur functions form a linear basis for the space of all symmetric polynomials \cite{stanley_fomin_1999}. We can therefore express the skew Schur function in Theorem \ref{theorem1} as a linear combination of Schur functions.
\begin{corollary}
\label{schurcorol}
The generating function $G_{(u,v)}^{w,\alpha,\beta}(t)$ of directed weighted paths is given by
\[
G_{(u,v)}^{w,\alpha,\beta}(t)=(-1)^{1-\alpha}\frac{1}{tp_{\alpha}}\frac{\sum\limits_{l=0}^{r}s_{(w^{\alpha-1},w-(v-u)_+-l,(u-v)_++l,0^{\beta-1})}(\bar{z})}{s_{((w+1)^{\alpha},0^{\beta})}(\bar{z})}\;,
\]
where $r=\min(u,v,w-u,w-v)$.
\end{corollary}

At this point we should like to remark that numerical experimentation with Maple led us to conjecture Corollary \ref{schurcorol} first, however we did not find a direct proof that avoided skew Schur functions.

Excursions, bridges, and meanders are all contained as special cases. For excursions we recover the result given in \cite{Bousquet2006},
\begin{equation}
G_{(0,0)}^{w,\alpha,\beta}(t)=G_{(w,w)}^{w,\alpha,\beta}(t)=\frac{(-1)^{1-\alpha}}{tp_{\alpha}}\frac{s_{(w^{\alpha},0^{\beta})}(\bar{z})}{s_{((w+1)^{\alpha},0^{\beta})}(\bar{z})}\;,
\end{equation}
and for bridges we find
\begin{equation}
G_{(0,w)}^{w,\alpha,\beta}(t)=\frac{(-1)^{1-\alpha}}{tp_{\alpha}}\frac{s_{(w^{\alpha-1},0^{\beta+1})}(\bar{z})}{s_{((w+1)^{\alpha},0^{\beta})}(\bar{z})}
\end{equation}
and
\begin{equation}
G_{(w,0)}^{w,\alpha,\beta}(t)=\frac{(-1)^{1-\alpha}}{tp_{\alpha}}\frac{s_{(w^{\alpha+1},0^{\beta-1})}(\bar{z})}{s_{((w+1)^{\alpha},0^{\beta})}(\bar{z})}\;,
\end{equation}
which are related by obvious symmetry. Similarly, for meanders we find
\begin{equation}
G_{(0,v)}^{w,\alpha,\beta}(t)=\frac{(-1)^{1-\alpha}}{tp_{\alpha}}\frac{s_{(w^{\alpha-1},w-v,0^{\beta})}(\bar{z})}{s_{((w+1)^{\alpha},0^{\beta})}(\bar{z})}\;,
\end{equation}
\begin{equation}
G_{(w,v)}^{w,\alpha,\beta}(t)=\frac{(-1)^{1-\alpha}}{tp_{\alpha}}\frac{s_{(w^{\alpha},w-v,0^{\beta-1})}(\bar{z})}{s_{((w+1)^{\alpha},0^{\beta})}(\bar{z})}\;,
\end{equation}
\begin{equation}
G_{(u,w)}^{w,\alpha,\beta}(t)=\frac{(-1)^{1-\alpha}}{tp_{\alpha}}\frac{s_{(w^{\alpha-1},u,0^{\beta})}(\bar{z})}{s_{((w+1)^{\alpha},0^{\beta})}(\bar{z})}\;,
\end{equation}
\begin{equation}
G_{(u,0)}^{w,\alpha,\beta}(t)=\frac{(-1)^{1-\alpha}}{tp_{\alpha}}\frac{s_{(w^{\alpha},u,0^{\beta-1})}(\bar{z})}{s_{((w+1)^{\alpha},0^{\beta})}(\bar{z})}\;.
\end{equation}

We prove Theorem \ref{theorem1} and Corollary \ref{schurcorol} in a sequence of steps in Section 2. Section 3 contains examples of some specific step sets.

\section{Proofs}

We consider the generating function
\begin{equation}
G(t,z)=\sum_{v=0}^{w}G_{(u,v)}^{w,\alpha,\beta}(t)z^v\;,
\label{gf}
\end{equation}
where for convenience we drop the indices $w$, $\alpha$, $\beta$ and $u$ on the left-hand side. We present a functional equation satisfied by $G(t,z)$ and define the notion of the kernel for this functional equation (this is $K(t,z)$ in the statement of Theorem \ref{theorem1}), which up to a prefactor is a polynomial in $z$ of degree $\alpha+\beta$. Coefficients of the kernel can be interpreted in terms of elementary symmetric functions of the roots, which will be central in our approach. The functional equation is  equivalent to setting up a system of linear equations, and using elementary symmetric functions will allow us to employ the Jacobi-Trudi formula to express the solution of the system in terms of skew Schur functions, leading to the expression in Theorem \ref{theorem1}.
\begin{proposition}
The generating function $G(t,z)$ satisfies the functional equation
\begin{multline}
G(t,z)=z^u+t\left(\sum_{a \in A} p_{a} z^a+\sum_{b \in B}\frac{q_{b}}{z^b}\right)G(t,z)\\-t\sum_{j=1}^{\infty}z^{w+j}\sum_{a \geq j}p_aG_{(u,w-a+j)}(t)-t\sum_{j=1}^{\infty}z^{-j}\sum_{b \geq j}q_bG_{(u,b-j)}(t)\;,
\label{feq}
\end{multline}
where $G_{(u,v)}(t)=G_{(u,v)}^{w,\alpha,\beta}(t)$. 
\end{proposition}
\begin{proof}
An $n$-step walk is constructed by adding steps from the step set $S$ to an $(n-1)$-step walk, provided $n>0$.
The zero step walk starting and ending at height $u$ is represented by $z^u$. The term $t\left(\sum_{a \in A} p_{a} z^a+\sum_{b \in B}\frac{q_{b}}{z^b}\right)G(t,z)$ corresponds to steps appended without the consideration of violation of boundaries. The steps not allowed are removed by subtracting the terms which account for the steps crossing the strip boundaries at $y=0$ and $y=w$. More precisely, $t\sum_{j=1}^{\infty}z^{w+j}\sum_{a \geq j}p_aG_{(u,w-a+j)}(t)$ corrects overcounting by steps going above the line $y=w$, and $t\sum_{j=1}^{\infty}z^{-j}\sum_{b \geq j}q_bG_{(u,b-j)}(t)$ corrects overcounting by steps going below the line $y=0$.
\end{proof}
Next, we rearrange the functional equation as 
\begin{multline}
\left(1-t\sum_{a \in A} p_{a} z^a-t\sum_{b \in B}\frac{q_{b}}{z^b}\right)G(t,z)=\\z^u-t\sum_{j=1}^{\infty}z^{w+j}\sum_{a \geq j}p_aG_{(u,w-a+j)}(t)-t\sum_{j=1}^{\infty}z^{-j}\sum_{b \geq j}q_bG_{(u,b-j)}(t).
\label{kfeq}
\end{multline}
The prefactor of $G(z,t)$ in (\ref{kfeq}) is called the kernel of the functional equation,
\begin{equation}
K(t,z)=1-t\sum_{a \in A}p_az^a-t\sum_{b \in B}q_bz^{-b}.
\end{equation}
It will be convenient to relate the coefficients of the kernel to elementary symmetric functions.
\begin{lemma}
\label{lemma1}
The kernel $K(t,z)$ can be written as
\begin{equation}
K(t,z)=-tp_{\alpha}\sum_{i=0}^{\alpha+\beta}z^{\alpha-i}(-1)^{i}e_i(z_1,z_2, \ldots, z_{\alpha+\beta})
\end{equation}
where $\bar{z}=z_1,z_2, \ldots, z_{\alpha+\beta}$ are the roots of the kernel $K(t,z)$, and we have
\begin{equation}
-tp_a=-tp_{\alpha}(-1)^{\alpha-a}e_{\alpha-a}(\bar{z})
\label{p-val}
\end{equation}
\begin{equation}
1-tp_0=-tp_{\alpha} (-1)^{\alpha}e_{\alpha}(\bar{z})
\label{p0-val}
\end{equation}
\begin{equation}
-tq_b=-tp_{\alpha}(-1)^{\alpha+b}e_{\alpha+b}(\bar{z})
\label{q-val}
\end{equation}
for $1\le a\le\alpha$ and $1\le b\le\beta$.
\end{lemma}
\begin{proof}
Writing the kernel in terms of its roots $\bar z$ we get
\begin{equation}
K(t,z)
=-\frac{tp_{\alpha}}{z^{\beta}}\prod_{k=1}^{\alpha+\beta}(z-z_k)
=-tp_{\alpha}\sum_{i=0}^{\alpha+\beta}z^{\alpha-i}(-1)^{i}e_i(\bar{z})\;,
\end{equation}
where we have introduced the elementary symmetric functions $e_j$ \cite{macdonald1998symmetric} defined by
\begin{equation}
\prod_{k=1}^{n}(z+z_k)=\sum_{j=0}^{n}z^{n-j}e_{j}(\bar{z})\;.
\end{equation}
Comparing coefficients of
\begin{equation}
1-t\sum_{a=0}^{\alpha}p_az^a-t\sum_{b=1}^{\beta}q_bz^{-b}=-tp_{\alpha}\sum_{i=0}^{\alpha+\beta}z^{\alpha-i}(-1)^{i}e_i(\bar{z})
\label{kernel_e}
\end{equation}
for different powers of $z$ completes the proof. 
\end{proof}
\begin{proposition}
The functional equation \eqref{feq} is equivalent to 
\begin{equation}
\sum_{v=0}^{w} \left(\sum_{i=0}^{\alpha+\beta}(-1)^{i}e_iG_{(u,v-\alpha+i)}(t)\right)z^v =-\frac{z^u}{tp_{\alpha}}
\label{syseq}
\end{equation}
\end{proposition}
\begin{proof}
We aim to rewrite the functional equation \eqref{kfeq} in terms of elementary symmetric functions instead of weights $p_a, q_b$ and $t$. 
Using Lemma \ref{lemma1}, we find
\begin{multline}
\left(\sum_{i=0}^{\alpha+\beta}z^{\alpha-i}(-1)^{i}e_i\right)\sum_{v=0}^{w}G_{(u,v)}(t)z^v=
-\frac{z^u}{tp_{\alpha}}\\+\sum_{j=1}^{\infty}\sum_{a \geq j}z^{w+j}(-1)^{\alpha-a}e_{\alpha-a}G_{(u,w-a+j)}(t)\\
+\sum_{j=1}^{\infty}\sum_{b \geq j}z^{-j}(-1)^{\alpha+b}e_{\alpha+b}G_{(u,b-j)}(t)\;.
\label{stuff}
\end{multline}
We rewrite the left hand side of \eqref{stuff} as 
\begin{multline}
\left(\sum_{i=0}^{\alpha+\beta}z^{\alpha-i}(-1)^{i}e_i\right)\sum_{v=0}^{w}G_{(u,v)}(t)z^v
=\sum_{v=-\infty}^{\infty}\left(\sum_{i=0}^{\alpha+\beta}(-1)^{i}e_iG_{(u,v-\alpha+i)}(t)\right)z^{v},
\end{multline}
where we have extended the limits of summation over $v$, as $G_{(u,v)}(t)$ is zero if the end point of the path is outside the strip.
Careful inspection of \eqref{stuff} shows that all the powers of $z^v$ with $v<0$ and $v>w$ cancel,
and we are left with with the desired result.
\end{proof}
The boundary corrections in the functional equation have of course been introduced to precisely that effect, as they were added to correct for steps that went beyond the upper and lower boundaries.
\begin{proof}[Proof of Theorem \ref{theorem1}]
Comparing coefficients of $z^v$ for $0 \leq v \leq w$, equation \eqref{syseq} is equivalent to a system of $w+1$ equations given by
\begin{equation}
\tilde{A}x=b\;,
\label{meq}
\end{equation}
where $x$ is the vector of unknowns $G_{(u,v)}(t)$, and $b$ is the column vector on the right hand side with a single non zero entry $-\frac{1}{tp_{\alpha}}$. 
Using the convention that $e_k=0$ if $k<0$ or $k>\alpha+\beta$, $\tilde{A}$ is the coefficient matrix 
\begin{equation}
\tilde{A}=\left((-1)^{\alpha+j-i}e_{\alpha+j-i}\right)_{i,j=0}^w\;,
\end{equation}
so that the non zero entries of $\tilde A$ form a diagonal band. We can evaluate the unknowns $G_{(u,v)}(t)$ for $v=0 \ldots w$, by using Cramer's rule. Before we do this, we first remove the negative signs of the entries in $\tilde{A}$ to write \eqref{meq} in terms of the matrix
\begin{equation}
A=\left(e_{\alpha+j-i}\right)_{i,j=0}^w=\begin{bmatrix}
e_{\alpha} & e_{\alpha+1}& e_{\alpha+2} & \cdots & e_{\alpha+\beta} & \cdots & 0 \\ 
e_{\alpha-1} & e_{\alpha} & e_{\alpha+1} & \cdots & e_{\alpha+\beta-1} & \cdots & 0 \\
e_{\alpha-2} & e_{\alpha-1} & e_{\alpha} & \cdots & e_{\alpha+\beta-2} & \cdots & 0 \\ 
\vdots & \vdots & \vdots & \vdots & \vdots & \ddots & \vdots \\
e_0 & e_1 & e_2 & \cdots & e_{\beta}& \cdots & 0 \\ 
\vdots & \vdots & \vdots & \vdots & \vdots & \ddots & \vdots \\
0 & 0 & 0 & 0 &0& \cdots &e_{\alpha}
\end{bmatrix}\;.
\label{amatrix}
\end{equation}
We accomplish this by applying a transformation given by the diagonal matrix $S$ with entries $(S)_{ii}=(-1)^i$ for $0\leq i\leq w$.
The matrix equation \eqref{meq} will be transformed as 
$S\tilde{A}S^{-1}Sx=Sb$.
We note that $S\tilde{A}S^{-1}=(-1)^{\alpha}A$ and $Sb=(-1)^ub$, so we have
\begin{equation}
(-1)^{\alpha}A(Sx)=(-1)^ub,
\label{tmateq}
\end{equation}
where 
$(Sx)_k=(-1)^k G_{(u,k)}(t)$.
To evaluate $Sx$ using Cramer's rule, let $A_{(u,v)}$ be the matrix formed by replacing column $v$ in $A$ with the column vector $(-1)^\alpha Sb$, which has  $(-1)^{u+1-\alpha}\frac{1}{tp_{\alpha}}$ at position $u$. so that
\begin{equation}
(-1)^vG_{(u,v)}(t)=\frac{|A_{(u,v)}|}{|A|}\,.
\end{equation}
What is left is to compute the determinants $|A|$ and $|A_{(u,v)}|$.
Using the second Jacobi -Trudi formula \cite{macdonald1998symmetric}, which 
expresses the Schur function as a determinant in terms of the elementary symmetric functions as 
\begin{equation}
s_{\lambda}=\det(e_{\lambda_i'+j-i})_{i,j=1}^{l(\lambda')}\;,
\label{jtfs}
\end{equation}
where $\lambda'$ is the partition conjugate to $\lambda$,
we can write $|A|$ in terms of a Schur function $s_\lambda$.
Comparing the determinant in \eqref{jtfs} with the matrix $A$ in \eqref{amatrix}, we can see that  the conjugate partition $\lambda'$  is given by
\begin{equation}
\lambda'=\left(\alpha^{w+1}\right).
\end{equation}
From this we can write $\lambda=((w+1)^{\alpha},0^\beta)$ and so $|A|$ can be written as 
\begin{equation}
|A|=s_{((w+1)^{\alpha},0^{\beta})}(z_1,z_2,\cdots,z_{\alpha+\beta}).
\label{dv}
\end{equation}
Note that we have chosen the convention to let the partition have the same number of parts as we have roots $z_1,z_2,\cdots,z_{\alpha+\beta}$, so that we supplement the partition with zero size parts as needed.

To evaluate the determinant of the matrix $A_{(u,v)}$, we make use of the fact that the only non zero entry in the $v$-th column is  $(-1)^{u+1-\alpha}\frac{1}{tp_{\alpha}}$ and
expand the determinant by that column to get 
\begin{multline}
|A_{(u,v)}|=\frac{(-1)^{2u+v+1-\alpha}}{tp_{\alpha}}\\
\times
\begin{vmatrix}
e_{\alpha}  & \cdots & e_{\alpha+v-1} & e_{\alpha+v+1} & \cdots&  e_{\alpha+\beta} & \cdots & 0 \\ 
e_{\alpha-1}   & \cdots & e_{\alpha+v-2} & e_{\alpha+v} & \cdots & e_{\alpha+\beta-1} & \cdots & 0 \\
e_{\alpha-2}   & \cdots & e_{\alpha+v-3} & e_{\alpha+v-1} & \cdots & e_{\alpha+\beta-2} & \cdots & 0 \\ 
\vdots  & \vdots  & \vdots & \vdots & \ddots & \vdots& \ddots & \vdots \\
e_{\alpha-u+1}   & \cdots & e_{\alpha+v-u} & e_{\alpha+v-u+2} & \cdots & e_{\alpha+\beta-u+1} & \cdots & 0 \\
e_{\alpha-u-1}   & \cdots & e_{\alpha+v-u-2} & e_{\alpha+v-u-1} & \cdots & e_{\alpha+\beta-u-1} & \cdots & 0 \\
e_{\alpha-u-2}   & \cdots & e_{\alpha+v-u-3} & e_{\alpha+v-u-1} & \cdots & e_{\alpha+\beta-u-2} & \cdots & 0 \\  
\vdots & \vdots  & \vdots  & \vdots & \ddots & \vdots & \ddots & \vdots \\
e_0  & \cdots & e_{v-1} & e_{v+1}& \cdots & e_{\beta} & \cdots & 0 \\ 
\vdots  & \vdots & \vdots & \vdots & \ddots & \vdots & \ddots & \vdots \\
0 & 0  & 0 &0 & \cdots &0 & \cdots & e_{\alpha}
\end{vmatrix}_{[w]}\;,
\label{remm}
\end{multline}
where the indices of $e_k$ increase by $2$ from the $(v-1)$-st to the $v$-th column.

Using the second Jacobi -Trudi formula \cite{macdonald1998symmetric} for skew Schur functions,
\begin{equation}
s_{\lambda/\mu}=\det(e_{\lambda_i'-\mu_j'+j-i})_{i,j=1}^{l(\lambda')}\;,
\label{ssjtf}
\end{equation}
we can express the determinant in \eqref{remm} by a skew Schur function. We find  
\[\lambda'=(\alpha+1^u,\alpha^{w-u})\quad\mbox{and}\quad\mu'=(1^v,0^{w-v})\;,\] 
and hence 
\begin{equation}
\lambda=(w^{\alpha},u,0^{\beta-1})
\quad\mbox{and}\quad
\mu=(v,0^{\alpha+\beta-1}),
\end{equation}
where we have again added zero size parts to follow the convention established above.

\begin{figure}[h!]
\centering
\includegraphics[width=0.5\textwidth]{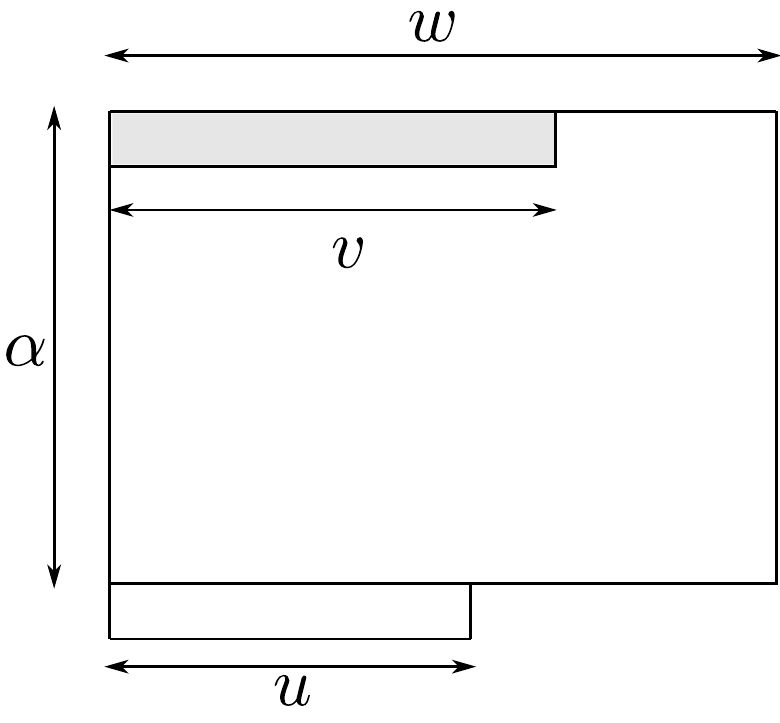}
\caption{Skew partition $\lambda/\mu=(w^{\alpha},u,0^{\beta-1})/(v,0^{\alpha+\beta-1})$ for the skew Schur function related to $\det A_{(u,v)}$. Here and in what follows we employ the `British' convention that the parts of the partition are depicted such the the largest part is at the top and the smallest one at the bottom. Note that we do not show parts of zero size.}
\label{figurethatneedstobementioned}
\end{figure}

A pictorial representation of the skew partition is given in Figure \ref{figurethatneedstobementioned}. We see that the associated skew partition is given by a rectangle of size $w\times\alpha$ which has a row of size $u$ added below and a row of size $v$ removed from the top row.
The corresponding skew Schur function is
$s_{(w^{\alpha},u,0^{\beta-1})/(v,0^{\alpha+\beta-1})}(\bar{z})$,
and therefore
\begin{equation}
|A_{(u,v)}|=(-1)^{v+1-\alpha}\frac{1}{tp_{\alpha}}s_{(w^{\alpha},u,0^{\beta-1})/(v,0^{\alpha+\beta-1})}(\bar{z})\;.
\end{equation}
Together with the expression of $|A|$ from \eqref{dv} we can write that $G_{(u,v)}$ is given by
\begin{equation}
(-1)^vG_{(u,v)}(t)=\frac{|A_{(u,v)}|}{|A|}=(-1)^{v+1-\alpha}\frac{1}{tp_{\alpha}}\frac{
s_{(w^{\alpha},u,0^{\beta-1})/(v,0^{\alpha+\beta-1})}(\bar{z})}{s_{((w+1)^{\alpha},0^\beta)}(\bar{z})}\;,
\end{equation}
which gives \eqref{skewresult} as needed.
\end{proof}
To prove Corollary \ref{schurcorol}, we need a technical lemma expanding the skew Schur function occurring in Theorem \ref{theorem1} in terms of Schur functions.
\begin{lemma}
\label{lemma}
Let $\alpha,\beta,w>0$. Then for $0\leq u,v\leq w$ we have
\begin{multline}
s_{(w^{\alpha},u,0^{\beta-1})/(v,0^{\alpha+\beta-1})}(z_1,\ldots,z_{\alpha+\beta})=\\
\sum_{l=0}^{r}s_{(w^{\alpha-1},w-(v-u)_+-l,(u-v)_++l,0^{\beta-1})}(z_1,\ldots,z_{\alpha+\beta}),
\label{skewtoschur}
\end{multline}
where $r=\min(u,v,w-u,w-v)$. 
\end{lemma}
\begin{proof}
From Pieri's rule \cite[Corollary 7.15.9]{stanley_fomin_1999}, we know that for a skew partition $\lambda/\nu$, where $\nu$ is a single-part partition $(v)$,
\begin{equation}
s_{\lambda/\nu}(z)=\sum _{\mu}s_{\mu}(z),
\label{cor7159}
\end{equation}
where the sum ranges over all partitions $\mu \subseteq \lambda$ for which $ \lambda/ \mu$ is a partition with one part of size $v$. In order to prove this lemma we specify the partitions $\lambda$ and $\nu$ as on the left hand side of \eqref{skewtoschur}. The partitions associated with the skew Schur function are 
\begin{equation}
\lambda=(w^{\alpha},u,0^{\beta-1})
\end{equation} 
and
\begin{equation}
\nu=(v,0^{\alpha+\beta-1}).
\end{equation} 
\begin{figure}[h!]
\centering
\includegraphics[width=0.5\textwidth]{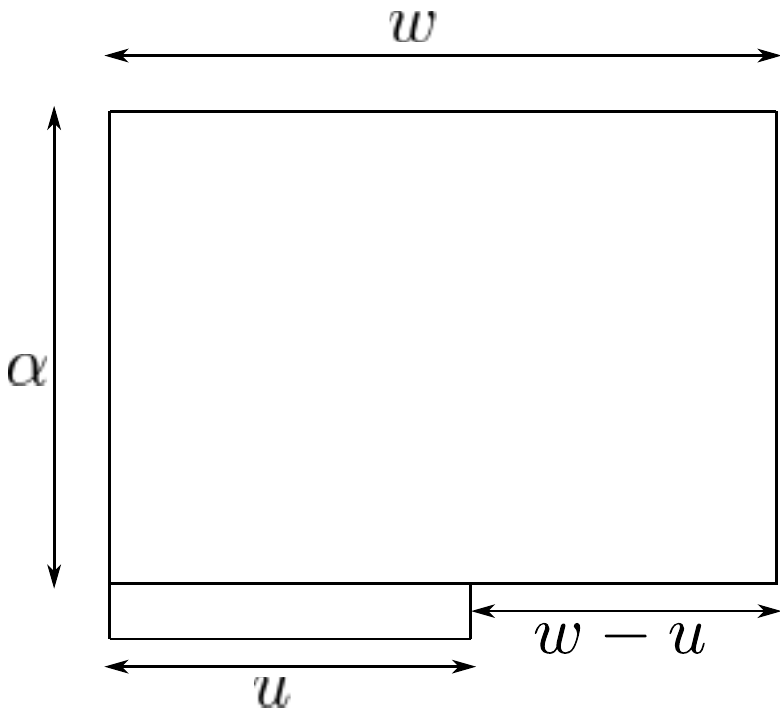}
\caption{A diagram of the partition $\lambda=(w^{\alpha},u,0^{\beta-1})$ occurring in the identity \eqref{cor7159}.}
\label{lpartition}
\end{figure}\\
The aim is to find an explicit expression for all partitions $\mu$ in the sum on the right hand side  of \eqref{cor7159}.
Given a partition $\lambda$ of the shape depicted Figure \ref{lpartition}, we want to find all partitions $\mu$ for which $\lambda/\mu$ is a horizontal strip of size $v$. This can be viewed as removing a strip of size $v$ from $\lambda$ so that the remaining object is still a valid partition. This removal can only be done from the last two rows, as removing anything from above the last two rows will not correspond to the removal of a strip. As the bottom row is of size $u$, the options of removing a strip of size $v$ depend on the size of $u$ and $v$. For this we consider two cases depending on whether the size $v$ of the strip to be removed exceeds the length $u$ of the bottom row or not.
\subsubsection*{Case $u \leq v$:}
\begin{figure}[h!]
\centering
\includegraphics[width=0.5\textwidth]{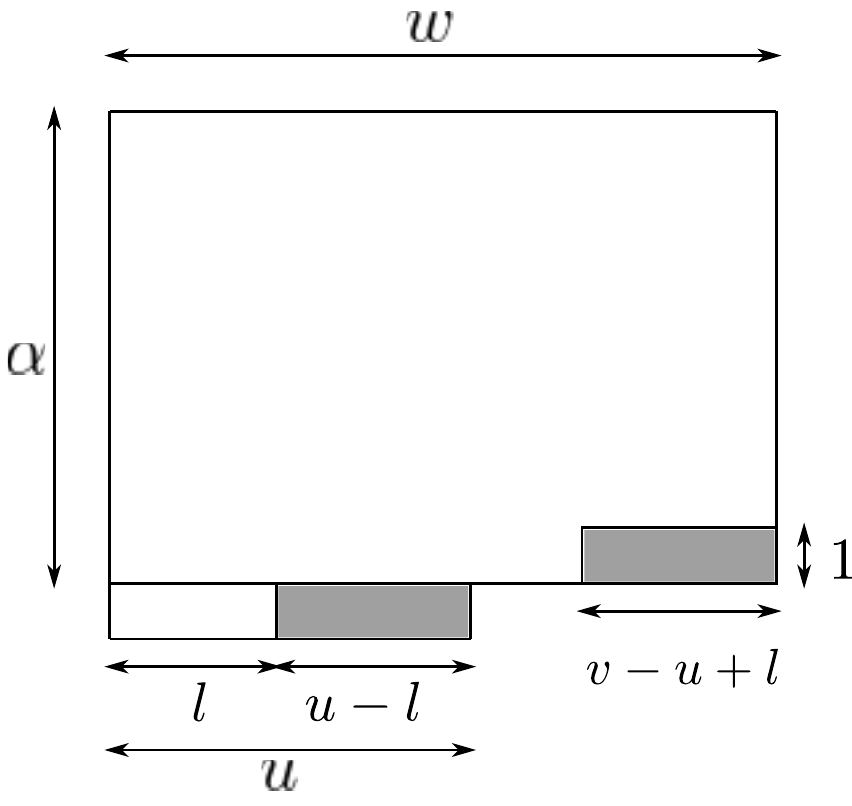}
\caption{A diagram showing the structure of the partition $\mu=(w^{\alpha-1},w-(v-u)-l,l)$ in the case $u\leq v$. The shaded part corresponds to a strip of size $v$.}
\label{proofpic}
\end{figure}
Consider a skew partition $\lambda/\nu$ where $\lambda=(w^{\alpha},u,0^{\beta-1})$ and $\nu={(v,0^{\alpha+\beta-1})}$ as shown in Figure \ref{figurethatneedstobementioned}. If $u\leq v$ then the structure of the partitions $\mu$ appearing in the sum on the right hand side  of \eqref{cor7159} is indicated in Figure \ref{proofpic}. The shaded portion shows the strip $\nu$ to be removed. 
We remove part of $\nu$ from the bottom row of length $u$ and the remaining part from the row above, i.e. we shorten the bottom row by $u-l$ and the row above by $v-u+l$. 
Removing the strip $\nu$ from $\lambda$ gives the following partition
\begin{equation}
\mu=(w^{\alpha-1},w-(v-u)-l,l)\;.
\end{equation}
Here, $l$ is constrained by
$$ l\leq \min(u,w-v).$$
Remembering that $u\leq v$, the sum can therefore be written as claimed,
\begin{equation}
s_{(w^{\alpha},u,0^{\beta-1})/(v,0^{\alpha+\beta-1})}=\sum_{l=0}^{\min(u,v,w-u,w-v)}s_{(w^{\alpha-1},w-(v-u)-l,l,0^{\beta-1})}\;.
\label{schurulv}
\end{equation}
\subsubsection*{Case $u > v$:}
\begin{figure}[h!]
\centering
\includegraphics[width=0.5\textwidth]{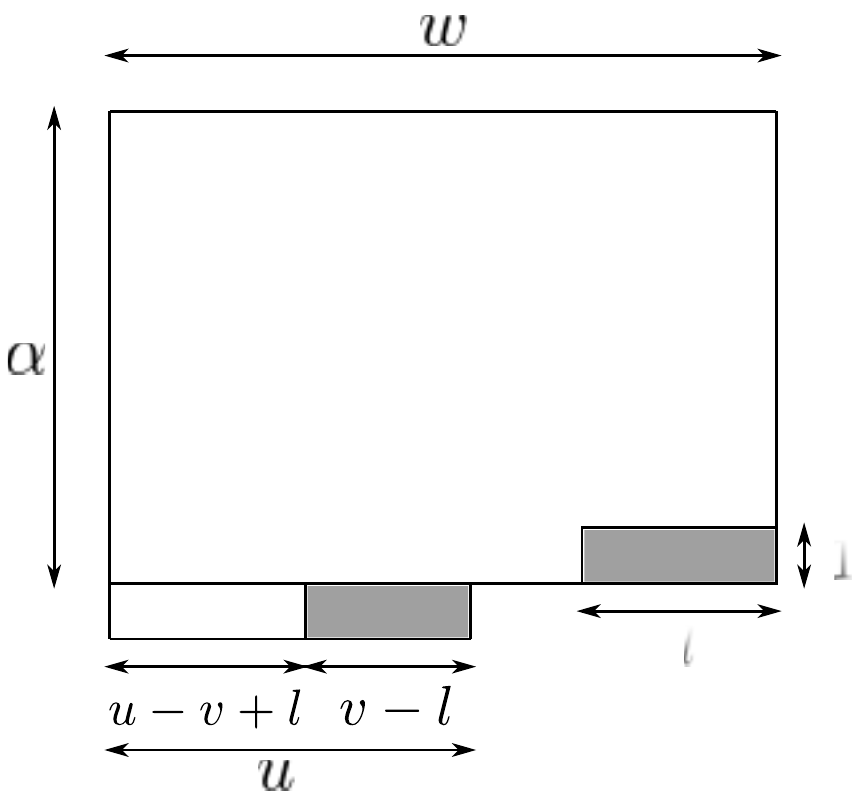}
\caption{A diagram showing the structure of the partition $\mu=(w^{\alpha-1},w-l,u-v+l)$ in the case $u>v$. The shaded part corresponds to a strip of size $v$.}
\label{proofpic2}
\end{figure}
We use the same idea as in the first case and remove strip $\nu$ from the partition $\lambda$. For $v<u$ the structure of the partitions $\mu$ appearing in the sum on the right hand side of \eqref{cor7159} are indicated in Figure \ref{proofpic2}.
Since $v<u$, we can remove $\nu$ completely from the lowest row and nothing from the row above, or we can remove part of it from the lowest row and the rest from the row above. We thus shorten the bottom row by $v-l$ and the row above by $l$. 
Removing the strip $\nu$ from $\lambda$ therefore gives the partition
\begin{equation}
\mu=(w^{\alpha-1},w-l,u-v+l)\;.
\end{equation}
Here, $l$ is constrained by $$ l\leq \min(v,w-u).$$
Remembering that $u>v$, the sum can therefore also be written as claimed,
\begin{equation}
s_{(w^{\alpha},u,0^{\beta-1})/(v,0^{\alpha+\beta-1})}=\sum_{l=0}^{\min(u,v,w-u,w-v)}s_{(w^{\alpha-1},w-l,u-v+l,0^{\beta-1})}
\label{schurvlu}
\end{equation}
Taken together, \eqref{schurulv} and \eqref{schurvlu} prove the lemma.
\end{proof}
We now use this Lemma to state the desired equivalent result for Theorem \ref{theorem1} in terms of Schur functions. Note that while in Lemma \ref{lemma} we did not need to specify the arguments of the functions, here it is important that the arguments are given by the kernel roots.

\begin{proof}[Proof of Corollary \ref{schurcorol}]
Lemma $\ref{lemma}$ proves the corollary.
\end{proof}

\section{Examples}

We now present several special cases involving small values of $\alpha$ and $\beta$. The first case we examine is 
$(\alpha,\beta)=(1,1)$, which corresponds to weighted Motzkin paths, and also includes Dyck paths as a special case, if the weight of the horizontal step is set to $p_0=0$. This has been studied previously \cite{JOSUATVERGES20112064} \cite{CHEN2008329}, but the Schur function approach used here is different and focusses more on the structure of the problem than just giving explicit generating functions. We then examine the cases $(\alpha,\beta)=(1,2)$ and $(\alpha,\beta)=(2,1)$, the solution of which involves roots of cubic equations. Here, the strength of our Schur function approach becomes apparent, as any explicit solution involves cumbersome algebraic expressions. 

\subsection{Motzkin paths}

Theorem \ref{theorem1} shows that the geometric structure of the problem is encoded in the partition shapes, while the step weights are ``hidden'' in the kernel roots. For Motzkin paths the result is particularly simple and elegant, involving only partitions with two parts,
\begin{equation}
G_{(u,v)}^{w,1,1}(t)
=
\frac{1}{tp_1}\frac{s_{(w,u)/(v,0)}(z_1,z_2)}{s_{(w+1,0)}(z_1,z_2)}\;.
\label{motzskewschur}
\end{equation}
From a computational point of view, skew Schur functions are of course not that easy to evaluate, but with the help of Corollary \ref{schurcorol} we are able to state the result in terms of Schur functions,
\begin{equation}
G_{(u,v)}^{w,1,1}(t)
=
\frac{1}{tp_1}\frac{\sum\limits_{l=0}^{r}s_{(w-(v-u)_+-l,(u-v)_++l)}(z_1,z_2)}{s_{(w+1,0)}(z_1,z_2)}.
\label{motzschur}
\end{equation}
To expand the Schur functions we write them in terms of determinants. The Schur function in the denominator of Equation \eqref{motzschur} is given by
\begin{align}
s_{(w+1,0)}(z_1,z_2)=&\frac{1}{\Delta}\begin{vmatrix}
z_1^{w+2} & z_2^{w+2}\\
z_1^0 & z_2^0
\end{vmatrix}\\
=&\frac{1}{\Delta}
(z_1^{w+2}-z_2^{w+2}).
\end{align}
where $\Delta=\Delta(z_1,z_2)=z_1-z_2$ comes from a Vandermonde determinant evalution. 
Similarly expressing the Schur function in the numerator of Equation \eqref{motzschur} as a determinant implies
\begin{multline}
s_{(w-(v-u)_+-l,(u-v)_++l)}(z_1,z_2)=\frac{1}{\Delta}\begin{vmatrix}
z_1^{w-(v-u)_+-l+1} & z_2^{w-(v-u)_+-l+1}\\
z_1^{(u-v)_++l} & z_2^{(u-v)_++l}
\end{vmatrix}\\
=\frac{1}{\Delta}(
z_1^{w-(v-u)_+-l+1}z_2^{(u-v)_++l} - z_2^{w-(v-u)_+-l+1}z_1^{(u-v)_++l}
).
\end{multline}
Now substituting the expansion of these Schur functions into \eqref{motzschur}, we finally obtain
\begin{equation}
G_{(u,v)}^{w,1,1}(t)=\frac{1}{tp_1}\dfrac{\sum\limits_{l=0}^{r}(
z_1^{w-(v-u)_+-l+1}z_2^{(u-v)_++l} - z_2^{w-(v-u)_+-l+1}z_1^{(u-v)_++l}
)}{
z_1^{w+2}-z_2^{w+2}}
\end{equation}
Here, $z_1=z_1(t)$ and $z_2=z_2(t)$ are the roots of the kernel $K(t,z)=1-tp_0-tp_1z-tq_1/z$, so that they can be explicitly given as solutions of the quadratic equation
\begin{equation}
z^2-\frac{1/t-p_0}{p_1}z+\frac{q_1}{p_1}=0\;.
\label{motzkernel}
\end{equation}

\subsection{Case ($\alpha=1$, $\beta=2$)}
Structurally, this case is rather similar to the preceding one, however the Schur functions now have as argument three kernel roots $z_1(t), z_2(t)$ and $z_3(t)$, which are the solution to the kernel equation given by
\begin{equation}
z^3-\frac{1/t-p_0}{p_1}z^2+\frac{q_1}{p_1}z+\frac{q_2}{p_1}=0\;,
\end{equation}
so that a general explicit solution would involve roots of a cubic equation.
Theorem \ref{theorem1} implies that
\begin{equation}
G_{(u,v)}^{w,1,2}(t)=\frac{1}{tp_{1}}\frac{s_{(w,u,0)/(v,0,0)}(z_1,z_2,z_3)}{s_{(w+1,0,0)}(z_1,z_2,z_3)}\;,
\label{motzskweschur12}
\end{equation}
and the result given in Corollary \ref{schurcorol} can be written as
\begin{equation}
G_{(u,v)}^{w,1,2}(t)=\frac{1}{tp_1}\frac{\sum\limits_{l=0}^{r}s_{(w-(v-u)_+-l,(u-v)_++l,0)}(z_1,z_2,z_3)}{s_{(w+1,0,0)}(z_1,z_2,z_3)}\;.
\label{motzschur12}
\end{equation}
We expand the Schur functions and write them in form of determinants. The Schur function in the denominator is given by
\begin{multline}
s_{(w+1,0,0)}(z_1,z_2,z_3)=\frac{1}{\Delta}\begin{vmatrix}
z_1^{w+3} & z_2^{w+3} & z_3^{w+3}\\
z_1^1 & z_2^1 & z_3^1\\
z_1^0 & z_2^0 & z_3^0
\end{vmatrix}\\
=\frac{1}{\Delta}
(z_1^{w+3}(z_2-z_3)-z_2^{w+3}(z_1-z_3)+z_3^{w+3}(z_1-z_2))\;,
\end{multline}
where $\Delta=(z_1-z_2)(z_1-z_3)(z_2-z_3)$ is again a Vandermonde determinant (which will however cancel out in the final result).
Similarly expressing the Schur function in the numerator as a determinant implies
\begin{multline}
s_{(w-(v-u)_+-l,(u-v)_++l,0)}(z_1,z_2,z_3)
=\\
\frac{1}{\Delta}\left(
z_1^{w-(v-u)_+-l+2}(z_2^{(u-v)_++l+1}-z_3^{(u-v)_++l+1})\right.\\
\hspace*{1.5cm}- z_2^{w-(v-u)_+-l+2}(z_1^{(u-v)_++l+1}-z_3^{(u-v)_++l+1})\\
\left. +z_3^{w-(v-u)_+-l+2}(z_1^{(u-v)_++l+1}-z_2^{(u-v)_++l+1})
\right).
\end{multline}
Now substituting the expansion of Schur functions in \eqref{motzschur12}, we obtain
\begin{multline}
G_{(u,v)}^{w,1,2}(t)=\\
\frac{1}{tp_1} 
\dfrac{\mathlarger{\mathlarger{\mathlarger{\sum}}}\limits_{l=0}^r\left(
   \splitdfrac{\splitdfrac{z_1^{w-(v-u)_+-l+2}(z_2^{(u-v)_++l+1}-z_3^{(u-v)_++l+1})}
                                   {- z_2^{w-(v-u)_+-l+2}(z_1^{(u-v)_++l+1}-z_3^{(u-v)_++l+1})}}{+z_3^{w-(v-u)_+-l+2}(z_1^{(u-v)_++l+1}-z_2^{(u-v)_++l+1})}
\right)}
{z_1^{w+3}(z_2-z_3)-z_2^{w+3}(z_1-z_3)+z_3^{w+3}(z_1-z_2)}.
\label{12final}
\end{multline}

\subsection{Case ($\alpha=2$, $\beta=1$)}
The kernel equation now leads to
\begin{equation}
z^3+\frac{p_1}{p_2}z^2-\frac{1/t-p_0}{p_2}z+\frac{q_1}{p_2}=0\;.
\end{equation}
We note that exchanging $\alpha$ and $\beta$ is akin to switching up and down steps with adjusting the weights appropriately. More precisely, making all the parameters explicit we have
\begin{equation}
K^{(2,1)}_{p_0,p_1,p_2,q_1}(t,z)=K^{(1,2)}_{p_0,q_1,q_2,p_1}(t,1/z)\;,
\end{equation}
which in the case of unit weights implies that the kernel roots for $(\alpha,\beta)=(2,1)$ and $(\alpha,\beta)=(1,2)$ are simply inverses of each other. This symmetry is not as explicit when writing the generating functions in terms of Schur functions. Symmetry considerations would dictate that we need to replace $u$ and $v$ by $w-u$ and $w-v$, respectively, but this is not obvious from the result given in Theorem \ref{theorem1}, which now reads
\begin{equation}
G_{(u,v)}^{w,2,1}(t)=-\frac{1}{tp_{2}}\frac{s_{(w,w,u)/(v,0,0)}(z_1,z_2,z_3)}{s_{(w+1,w+1,0)}(z_1,z_2,z_3)}\;.
\label{motzskweschur21}
\end{equation}
From Corollary \ref{schurcorol}, this can be written as
\begin{equation}
G_{(u,v)}^{w,2,1}(t)=-\frac{1}{tp_2}\frac{\sum\limits_{l=0}^{r}s_{(w,w-(v-u)_+-l,(u-v)_++l)}(z_1,z_2,z_3)}{s_{(w+1,w+1,0)}(z_1,z_2,z_3)}\;.
\label{motzschur21}
\end{equation}
We expand the Schur functions and write them in form of determinants,
and we obtain
\begin{multline}
G_{(u,v)}^{w,2,1}(t)=
-\frac{1}{tp_2}\times \\
\dfrac{\mathlarger{\mathlarger{\mathlarger{\sum}}}\limits_{l=0}^r\left(
   \splitdfrac{\splitdfrac{z_1^{w+2}(z_2^{w-(v-u)_+-l+1}z_3^{(u-v)_++l}-z_3^{w-(v-u)_+-l+1}z_2^{(u-v)_++l})}
                                   {- z_2^{w+2}(z_1^{w-(v-u)_+-l+1}z_3^{(u-v)_++l}-z_3^{w-(v-u)_+-l+1}z_1^{(u-v)_++l})}}{+z_3^{w+2}(z_1^{w-(v-u)_+-l+1}z_2^{(u-v)_++l}-z_2^{w-(v-u)_+-l+1}z_1^{(u-v)_++l})}
\right)}
{z_1^{w+3}(z_2^{w+2}-z_3^{w+2})-z_2^{w+3}(z_1^{w+2}-z_3^{w+2})+z_3^{w+3}(z_1^{w+2}-z_2^{w+2})}.
\label{21final}
\end{multline}
When written in terms of kernel roots, we see some structural similarity between \eqref{21final} and \eqref{12final}, in line with the symmetry observation made above. Obviously a more general study of the effect of symmetry would be an interesting topic for further work.

\end{document}